\documentclass[oneside,11pt]{amsart}
\usepackage[active]{srcltx}
\usepackage{amsmath,amssymb}
\newcommand{\vertiii}[1]{{\left\vert\kern-0.25ex\left\vert\kern-0.25ex\left\vert #1
    \right\vert\kern-0.25ex\right\vert\kern-0.25ex\right\vert}}
\usepackage{amsmath, amsfonts,amsthm,times,graphics}

 \makeatletter
\renewcommand*\subjclass[2][2000]{%
  \def\@subjclass{#2}%
  \@ifundefined{subjclassname@#1}{%
    \ClassWarning{\@classname}{Unknown edition (#1) of Mathematics
      Subject Classification; using '1991'.}%
  }{%
    \@xp\let\@xp\subjclassname\csname subjclassname@#1\endcsname
  }%
}
 \makeatother

\newtheorem{theorem}{Theorem}[section]

\theoremstyle{definition}

\newtheorem{remark}[theorem]{Remark}
\numberwithin{equation}{section}

\newcounter{minutes}
\divide\time by 60
\newcounter{hours}
\multiply\time by 60 \addtocounter{minutes}{-\time}

\begin{document}

\title[On Heinz type inequality  and Gaussian curvature of Minimal surfaces]{On Heinz type inequality for the half-plane and Gaussian curvature of Minimal surfaces}


\author{David Kalaj}
\address{Faculty of Natural Sciences and Mathematics, University of
Montenegro, Cetinjski put b.b. 81000 Podgorica, Montenegro}
\email{davidk@ucg.ac.me}

\def\thefootnote{}
\footnotetext{ \texttt{\tiny File:~\jobname.tex,
          printed: \number\year-\number\month-\number\day,
          \thehours.\ifnum\theminutes<10{0}\fi\theminutes }
} \makeatletter\def\thefootnote{\@arabic\c@footnote}\makeatother

\begin{abstract}
We prove a Heinz type inequality for harmonic diffeomorphisms of of the half-plane onto itself. We then apply this result to prove some sharp bound of the Gaussian curvature of a minimal surface, provided that it lies above the whole half-plane in $\mathbf{R}^3$.
\end{abstract}
\maketitle
\tableofcontents

\section{Introduction}

The aim of this note is to prove the following  results

\begin{theorem}\label{heinznew}
Assume that $f=h+\overline{g}$ is a harmonic diffeomorphsim of the half-plane $\mathbf{U}$ onto itself with $f(a)=b$. Then the following sharp inequality holds true \begin{equation}\label{heinz}|Df(z)|=|h'(z)|+|g'(z)|\ge \frac{\Im\ (b)}{\Im\ (a)} \  \ z\in \mathbf{U}.\end{equation} In particular if $f$ has a fixed point (for example $f(i)=i$), then \begin{equation}\label{heinz23}|Df(z)|=|h'(z)|+|g'(z)|\ge 1, \  \ z\in \mathbf{U} .\end{equation}
\end{theorem}

By taking the composition $F(z)=f(a(z))$, where $a(z)=i\frac{1+z}{1-z}$, is a conformal mapping of the unit disk onto the half-plane with $a(0)=i$, Theorem~\ref{heinznew} implies the following theorem.

\begin{theorem}\label{heinznew2}
Assume that $f=h+\overline{g}$ is a harmonic diffeomorphsim of the unit disk $\mathbf{D}$ onto the half-plane $\mathbf{U}$ with $f(0)=i$. Then the following sharp inequality holds true \begin{equation}\label{heinz}|Df(z)|=|h'(z)|+|g'(z)|\ge \frac{1}{2}\mathrm{dist}(f(0),\partial \mathbf{U}).\end{equation}
\end{theorem}
\begin{remark}\label{rem}
It follows from \cite[Theorem~2.2]{complex1}, that if instead of the half-plane $\mathbf{U}$, we consider an arbitrary convex domain $\Omega$, then we get the inequality \begin{equation}\label{heinz2}|Df(z)|=|h'(z)|+|g'(z)|\ge\frac{1}{4}\mathrm{dist}(f(0),\partial \Omega).\end{equation} An better inequality under some additional conditions has been obtained in \cite{geom}. We expect that in this contexts the constant $1/4$ in \eqref{heinz2} can be replaced by $1/2$. On the other hand Heinz in \cite{HE} proved that, if $\Omega=\mathbf{D}$ (i.e. if $\Omega$ is the unit disk) then instead of $1/4$ it can be taken $1/\pi$. We also conjecture that the right constant here is $2/\pi$. Finally,  Hall in \cite{hall1} (see as well \cite{hall2}) proved the sharp estimate   $|h'(0)|+|g'(0)|\ge \frac{3\sqrt{3}}{2\pi}$ from below for the harmonic diffeomorphisms of the unit disk onto itself fixing the origin.  Hall result gives so far the best bounds of the Gaussian curvature of the minimal surfaces at the point above the center of the unit disk, provided that the minimal surface is lifter from the unit disk. The obtained constants are however not sharp, and this problem remains an open challenging problem.
\end{remark}

We say that a minimal surface $\Sigma$ is lying over a whole halp-plane $\Pi$, if its orthogonal projection to $\Pi$ is a homeomorphism of $\Sigma$ onto $\Pi$.

By using Theorem~\ref{heinznew} we present a different proof of the following theorem by Schober and Hengartner (\cite{hs})
\begin{theorem}\label{quan}
Let $\Sigma$ be a minimal surface lying over a whole half-plane $\Pi$, whose boundary is the line $L$. Let $\zeta\in \Sigma$ and let $z=z(\zeta)$ be its (orthogonal) projection to $\Pi$. If $K(\zeta)$ is the Gaussian curvature of $\Pi$ at $\zeta$, then the sharp inequality \begin{equation}\label{minimalsharp}
K(\zeta)\le \frac{1}{\mathrm{dist}^2(z(\zeta),L)}
\end{equation}
holds for every $\zeta$.
\end{theorem}

\section{Preliminaries}
\subsection{Weierstrass--Enneper parameterization
of minimal surface}
The projections of minimal graphs in isothermal parameters are precisely
the harmonic mappings whose dilatations are squares of meromorphic functions.
 If $\Sigma$
is a minimal surface lying over a simply connected domain $\Omega$ in
the $uv$ plane, expressed in isothermal parameters ($x$,
$y$), its projection onto
the base plane may be interpreted as a harmonic mapping $w =
f (z)$, where
$w =u +iv $ and $z =
x +iy.$ After suitable adjustment of parameters, it may be
assumed that $f$ is a sense-preserving harmonic mapping of the  $\mathbf{U}$
onto $\Omega$, with $f (i) =w_0$ for some preassigned point $w_0$ in $\Omega$. Let $f =
h +
\bar g$
be the canonical decomposition, where $h$ and $g$ are holomorphic. Then the dilatation
$\mu
=\frac{g'}{h'}$
of $f$ is an analytic
 function with  $|\mu(z)|<1$ in $\mathbf{U}$
and with the further property that $\mu =q^2$
for some function $q$ analytic in $\mathbf{U}$. The minimal surface $\Sigma$
over $\Omega$
has the
isothermal representation $F=(u,v,t)$:

$$ u = \Re f (z) = \Re \int_i^z \phi_1(\zeta)d\zeta,$$

$$ v = \Im f (z) = \Im \int_i^z \phi_2(\zeta)d\zeta,$$

$$ t = \Im \int_i^z \phi_3(\zeta)d\zeta,$$

with \begin{equation}\label{phi12}\phi_1 =h'+ g'=p(1 + q^2),\ \phi_2 = -i(h'-
g') =-ip(1 -q^2), \end{equation} and \begin{equation}\label{fi3}  \phi_3=
2ipq,\end{equation}
where $p$ and $q$ are the so-called Weierstrass-Enneper parameters. Thus \begin{equation}\label{pp}
h'= p.\end{equation}
The first fundamental form of $\Sigma$
is $$ds^2 =\lambda^2 |dz|^2,$$ where
$$\lambda^2(z)=\frac{1}{2}\sum_{1}^{3}|\phi_k|^2.$$
A direct calculation shows that
\begin{equation}\label{lambda}\lambda=|h'|+ |g'|=|p|(1 + |q|^2).\end{equation} For this fact and other important properties of minimal surfaces we refer to the book of Duren \cite{dure}.

\subsection{Gaussian curvature of Minimal Surfaces}
This simple expression \eqref{lambda} allows us to calculate the Gauss curvature of $S$ in
terms of the underlying harmonic mapping. Note that, by Lewy theorem $p(z) = h'(z) \neq  0$ in $\mathbf{U}$
since $f$ is sense-preserving. The general formula
for Gauss curvature is $$K =-\frac{\Delta \log \lambda}{\lambda^2}.$$

Therefore, in terms of the Weierstrass - Enneper parameters, the Gauss curvature
is found to be (cf. \cite{dure})
$$K = - \frac{4|q'|^2}{|p|^2(1 + |q|^2)^4}.$$
Since the underlying harmonic mapping $f$ has dilatation $\omega=g'/h'=q^2$
and $h' = p$, an equivalent expression is
$$K = -
\frac{|\omega'|^2}{|h'g'|(1+|\omega|)^4} .$$
The previous formula is suitable for using of analytic function theory to estimate
Gauss curvature.

Let $\eta$ be a conformal mapping of $\Omega$ onto $\mathbf{D}$, then the hyperbolic metric of $\Omega$ is given by $$\lambda_\Omega(z)=\frac{|\eta'(z)|}{1-|\eta(z)|^2}.$$

Since $|q(z)| <  1$, the Schwarz-Pick lemma gives
\begin{equation}\label{ineq}|q'(z)|
\le \lambda_\Omega(z)(1 - {|q(z)|^2})\ \ \  z \in\mathbf{U}.\end{equation}

Therefore, at the point of the surface that lies above $w = f (z)$, we get
the estimate
\begin{equation*}\begin{split}|K| &\le \lambda^2_\Omega(z) \frac{(1 - {|q(z)|^2})^2}
{|p(z)|^2(1 + |q(z)|^2)^4}
\\&= \lambda^2_\Omega(z)\frac{(1 - |\omega(z)|)^2}{
(|h'(z)| + |g'(z)|)^2(1 + |\omega(z)|)^2}
\\&\le \lambda^2_\Omega(z)\frac{ 1}
{(|h'(z)| + |g'(z)|)^2} .\end{split}\end{equation*}
Thus
\begin{equation}\label{impor}\begin{split}|K| \le  \frac{ \lambda^2_\Omega(z)}
{(|h'(z)| + |g'(z)|)^2} .\end{split}\end{equation}

If $\Omega$ is the unit disk then \begin{equation}\label{inedisk}\lambda_\Omega=\frac{1}{1-|z|^2}\end{equation} and thus
\begin{equation}\label{diskeq}|K| \le  \frac{ 4}
{(1-|z|^2)^2(|h'(z)| + |g'(z)|)^2} ,\ \ \  z \in\mathbf{U}.\end{equation}

The inequality  \eqref{diskeq} has been used by Hall see  \cite{hall2} (and in \cite{hall1}) to derive the bound in the unit disk setting (see Remark~\ref{rem}).
Now we assume that  $\Omega$ is the upper half-plane. Then\begin{equation}\label{ine}\lambda_\Omega=\frac{1}{2\Im(z)}\end{equation} and thus
\begin{equation}|K| \le  \frac{ 1}
{|\Im z|^2(|h'(z)| + |g'(z)|)^2}\ \ \  z \in\mathbf{U}.\end{equation}

\section{Proof of the main results}

\begin{proof}[Proof of Theorem~\ref{heinznew}]
Let $v(z)=\Im\, f(z)$. Then  $v$ is a positive harmonic function on
$\mathbf U$ and therefore, by the Riesz--Herglotz theorem (see \cite[Theorem~7.20]{ABR}), $v$ has the form $$ v(z)=cy+\pi^{-1}\int_{-\infty}^{+\infty}
P(z,t)\,d\mu(t), $$ where $c$ is a non-negative constant and $\mu$
is a non-decreasing function on $\mathbf R$ and $P$ is the Poisson
kernel, $$ P(z,t)=\frac{y}{|z-t|^2} \quad (z=x+iy\in\mathbf U, \
t\in\mathbf R).$$ Therefore \begin{equation}\label{twotwo}
\begin{split}
v(z)
&\ge cy+\pi^{-1}\int_x^{x+y}P(z,t)\,d\mu(t)\\ &\ge
cy+{\pi}^{-1}\int_x^{x+y}\frac{y}{2y^2}\,d\mu(t)\\ &=cy+
{\pi}^{-1}\frac{\mu(x+y)-\mu(x)}{y}\ge  0.
\end{split}
\end{equation}

Now assume that  $f$   is
continuos up to the boundary and in particular $v(x,y)\rightarrow
0$ as $y\rightarrow 0$ for any fixed $x\in\mathbf R.$ From this
and (\ref{twotwo}) it follows that the right derivative of $\mu$
vanishes everywhere. That the left derivative vanishes everywhere
can be proved in a similar way. Hence $\mu$ is constant, and this
proves that $v(z)=cy$   for some $c>0.$

Now $u(z)=\Re(f(z))=2\Re(k(z))$ for some holomorphic function $k$ defined on the upper half-plane. As $f=k(z)+cz+\overline{k(z)-c z}$ is locally univalent, by Lewy theorem, $$J(z,f)=|f_z|^2-|f_{\bar z}|^2=|k'(z)+c|^2-|k'(z)-c|^2=2c\,\Re (k'(z))>0.$$ Now if $a(z)=2k'(z)$, by taking into account the condition $f(i)=b$, we get $$f(z)=h+\bar g=\Re \int_i^ z a(z)dz+\Re(b)+i(\Im\ b) y$$ where $a$ is a holomorphic mapping of the upper halp-plane into the right-half plane. Then we have $$h'=\frac{1}{2}(a+\Im\ b), \ \  \ g'=\frac{1}{2}(a-\Im\ b).$$ So $$(|h'|+|g'|)^2=\frac{1}{4}(|a-\Im\ b|+|a+\Im\ b|)^2.$$
Then after some straight-forward calculations we get $$\frac{1}{4}(|a-\Im\ b|+|a+\Im\ b|)^2\ge \max\{(\Im\ b)^2,|a|^2\}\ge (\Im\ b)^2.$$

Assume now that  $f$ is not continuous up to the boundary. Then for $n\in\mathbf{N}$ let $\mathbf{U}_n=\{z\in\mathbf{U}: \Im \, z>\frac{1}{n}\}$ and let $\varphi_n(z)$ be a conformal mapping of $\mathbf{U}$ onto $f^{-1}(\mathbf{U}_n)$ so that $f(\varphi_n(i))=b+\frac{i}{n}$. Then the mapping $$f_n=\varphi_n(f(z))-\frac{i}{n}$$ is a harmonic diffeomorphism of $\mathbf{U}$ onto itself so that $f_n(i)=b$. Since $\varphi_n(z)$ converges in compacts subsets of $\mathbf{U}$ to the identity, its derivative converges in compacts subsets of $\mathbf{U}$ to the constant function $1$. So $$|Df(z)|=\lim_{n\to \infty}|Df_n(\varphi_n(z))||\varphi_n'(z)|\ge \Im(b)\lim_{n\to \infty}|\varphi_n(z)|=\Im(b).$$ This finishes the proof.
\end{proof}
\begin{proof}[Proof of Theorem~\ref{quan}]
\begin{figure}[htp]\label{f1}
\centering
\includegraphics{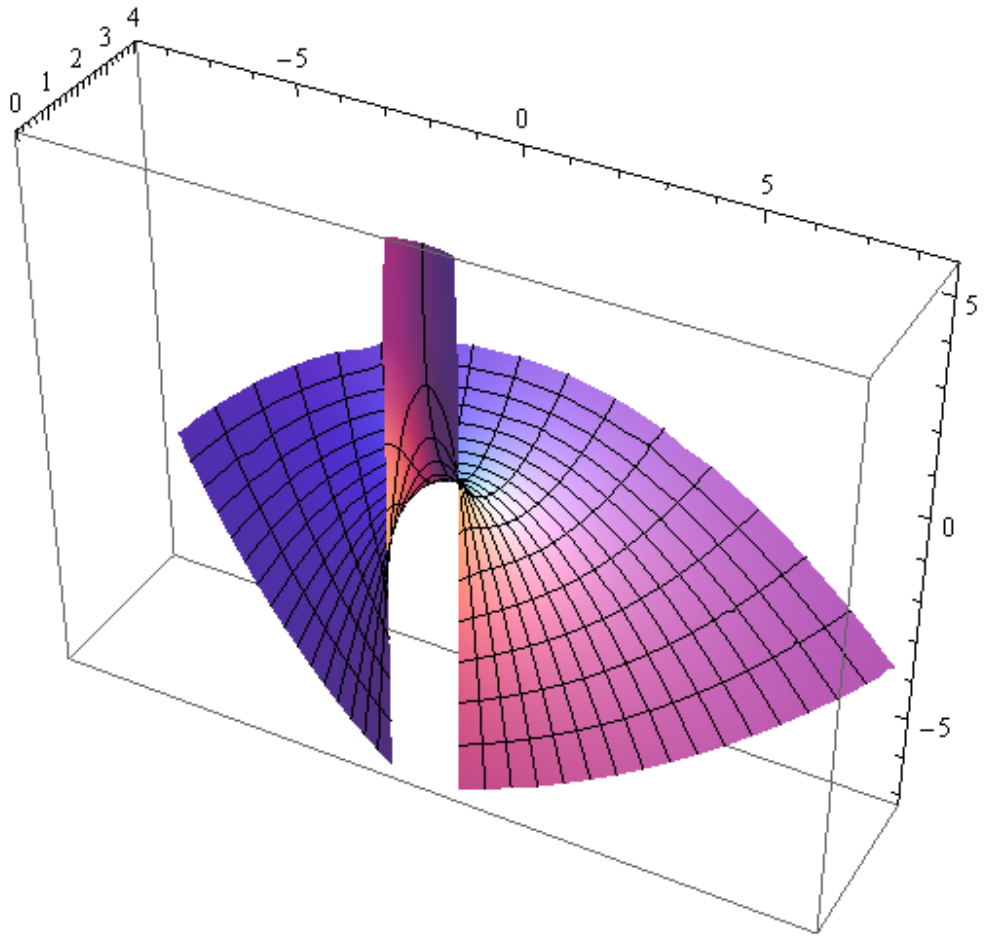}
\caption{The minimal surface over the half-plane.}
\end{figure}

Without loss of generality (after rotation if needed) we can assume that $$\Pi=\mathbf{R}^2\times {0}\cong \mathbf{C}.$$
We suppose that $\Omega = \mathbf{U}$ and $z=b\in \mathbf{U}$ is a fixed point.
Then $\Sigma$ is a minimal graph above the half-plane, and $K$ is the Gauss
curvature at the point on the surface above the basepoint $b$. The projection of $\Sigma$ is
then a harmonic mapping of $\mathbf{U}$ onto $\mathbf{U}$ with $f (i) = b$.
Further it can be assumed that $b=i$ so that $\mathrm{dist}(z,\mathbf{R})=1$.

By plugging $z=i$ and $f(i)=i$, where $i$ is the projection of $\zeta$ into  in \eqref{impor}, we get \begin{equation}\label{impor1}\begin{split}|K(\zeta)|   \le \frac{ 1}
{(|h'(i)| + |g'(i)|)^2}\le \frac{1}{1^2}=\frac{1}{\mathrm{dist}^2(b,\mathbf{R})} .\end{split}\end{equation} In order to show that the inequality is sharp, we make step by step analysis of the proof of our inequality.
Since $f=g+\overline{h}$ where $$h=\frac{1}{2}\left(\int_i^za(w)dw+z\right)$$ and $$g=\frac{1}{2}\left(\int_i^za(w)dw-z\right),$$ and since $$q(z)=\frac{z-i}{z+i}$$ is the only conformal mapping (up to the rotation) of the upper half-plane  onto the unit disk, so that $q(i)=0$ we need to solve the equation $$\frac{g'}{h'}=\frac{a-1}{a+1}=q^2(z).$$ After straight-forward calculation we get $$a=\frac{ z^2-1}{2iz}.$$
So
\begin{equation}\label{ff}f(z)=h(z)+\overline{g(z)}=\left(\Re\left[\frac{1+i  \pi +z^2-2 \log z}{4 i}\right]+i y\right).\end{equation}
where
$$h=\frac{m +z}{2}, \ \ \  g(z)=\frac{m- z}{2},$$ where $$m(z)= \frac{1+i  \pi +z^2-2 \log z}{4 i}.$$
Then $f$ maps the upper half-plane into itself and satisfies the condition $f(i)=b$ as well as \begin{equation}\label{shaka}f_z(i)=\frac{1}{2}(m'(i)+1)=1 \text {and}\ \ f_{\bar z}(i)=\frac{1}{2}(m'(i)-1)=0.\end{equation}

Since $$-2i pq=-2i h'\frac{z-i}{z+i}=2i\cdot\frac{i (i+z)^2}{4 z}\frac{z-i}{z+i}=-\frac{1+z^2}{2 z},$$ by \eqref{fi3},  the third coordinate of minimal surface laying above $\mathbf{U}$ is given by

\begin{equation}\label{tt}t(z)=-\Re\int_i^z\frac{1+\zeta^2}{2\zeta} d\zeta=\frac{1}{4} \left(-1-\Re\left[z^2\right]-2 \Re[\log z]\right).\end{equation}

The minimal surface $$\Sigma=\{(\Re\, f(z), \Im\,f(z), t(z)):z\in \mathbf{U}\}$$ is shown in Figure~1. It is a minimal surface over the halp-plane with the extremal gaussian curvature at the point above $i$, and it is \emph{the whole surface lying over $\mathbf{U}$}, because $f(\mathbf{U})=\left\{\left(\frac{1}{2} \left( x y+ \arctan \frac{x}{y}\right),y\right):y>0, x\in\mathbf{R}\right\}=\mathbf{U}.$

\end{proof}

\end{document}